\documentclass[11pt]{amsart}
\usepackage{amsmath,amssymb,amsthm}
\usepackage[margin=1.1in]{geometry}
\usepackage[colorlinks=true,linkcolor=blue,citecolor=blue]{hyperref}

\newtheorem{theorem}{Theorem}[section]
\newtheorem{lemma}[theorem]{Lemma}
\newtheorem{proposition}[theorem]{Proposition}
\newtheorem{corollary}[theorem]{Corollary}
\theoremstyle{definition}
\newtheorem{definition}[theorem]{Definition}
\newtheorem{example}[theorem]{Example}
\theoremstyle{remark}
\newtheorem{remark}[theorem]{Remark}

\newcommand{\sr}{\sim_{r}}
\newcommand{\ler}{\le_{r}}
\newcommand{\lr}{<_{r}}
\newcommand{\sMvN}{\sim}
\newcommand{\spec}{\mathrm{sp}}
\newcommand{\dist}{\mathrm{dist}}
\newcommand{\CC}{\mathbb{C}}
\newcommand{\Inv}{\mathrm{Inv}}

\newcommand{\Az}{A_{0}}
\newcommand{\g}{\mathfrak{g}}

\title[The pure infiniteness transfer problem]{The pure infiniteness transfer problem}
\author{Arindam Sutradhar}
\address{Theoretical Statistics and Mathematics Unit, ISI Banglore, Bengaluru, Karnataka 560059}
\email{arindam1050@gmail.com}
\date{\today}


\begin{document}

\begin{abstract}
Problem 8.4 of Pino, Goodearl, Perera and Molina \cite{AGPSM} asks whether a dense subalgebra $\Az$ of a $C^*$-algebra $A$ that is purely infinite as a ring forces $A$ to be purely infinite as a $C^*$-algebra. We call this the \emph{pure infiniteness transfer problem}, the transfer being from the ring to its $C^*$-completion. The problem is open, even when $\Az$ is unital and simple. We settle it under two hypotheses: $A$ has real rank zero, and $\Az$ is closed under holomorphic functional calculus. Under these hypotheses $\Az$ is purely infinite simple as a ring if and only if $A$ is purely infinite simple as a $C^*$-algebra. We also prove an obstruction: a unital $C^*$-algebra with a nonzero finite projection has no dense hfc-closed purely infinite simple unital subring. Finally, the hypotheses hold for proper subalgebras, for the gauge action of $\mathbb{T}$ on a Cuntz algebra $\mathcal{O}_n$, the smooth subalgebra $\mathcal{O}_n^\infty$ is a proper dense hfc-closed subalgebra that is purely infinite simple as a ring.

\end{abstract}

\maketitle
\section{Introduction}\label{sec:intro}

Cuntz's notion of a purely infinite $C^*$-algebra \cite{Cuntz77, Cuntz81} has a purely algebraic counterpart, introduced for rings by Ara, Goodearl and Pardo \cite{AGP}: a unital ring is \emph{purely infinite simple} if it is simple and every nonzero right ideal contains an infinite idempotent. The two notions are known to interact. A $C^*$-algebra that is purely infinite as a ring is purely infinite as a $C^*$-algebra \cite[Prop.~3.17]{AGPSM}. What is not known is whether the algebraic property survives the passage to a completion, that is, whether it passes from a dense subring to the ambient $C^*$-algebra. This is the content of the following problem of Aranda~Pino, Goodearl, Perera and Siles~Molina \cite[Problem~8.4]{AGPSM}.

\begin{quote}
\textbf{Problem 8.4.} \emph{Let $\Az$ be a dense subalgebra of a $C^*$-algebra $A$. If $\Az$ is purely infinite in the algebraic sense, is $A$ purely infinite in the $C^*$-algebraic sense?}
\end{quote}

\noindent We call this the \emph{pure infiniteness transfer problem}, the transfer being from the ring $\Az$ to its completion $A$; we abbreviate it to the transfer problem below. The problem is open. It is unresolved even when $\Az$ is a unital simple purely infinite ring \cite[Problem~8.6]{AGPSM}, and the same authors conjecture that a related converse fails even for $C^*$-algebras of real rank zero \cite[Rmk.~3.18]{AGPSM}. We do not resolve Problem 8.4; we prove it under an additional regularity hypothesis on $\Az$, and we explain below why some such hypothesis appears to be needed.

The reason the two theories are comparable at all is that pure infiniteness admits a formulation with no analysis in it. For a unital simple $C^*$-algebra $A$, pure infiniteness is equivalent to the requirement that $A \ne \CC$ and that every nonzero $a \in A$ satisfy
\[
xay = 1 \qquad \text{for some } x,y \in A ,
\]
and this is precisely the criterion that Ara, Goodearl and Pardo take as the definition for rings \cite[Thm.~1.6]{AGP}; see \cite{Abr} for a survey and \cite{AAS} for a systematic treatment. Pure infiniteness is therefore a statement about the multiplicative structure alone: no norm, no limit, no positivity is involved. That is what makes it a natural candidate for transfer across a dense inclusion, and it is what gives Problem 8.4 its interest.

The difficulty is that a dense subalgebra sees the multiplicative structure of $A$ but not its topology. The equation $xay = 1$ can be solved in $\Az$ for every $a \in \Az$, yet the solutions carry no information about elements of $A \setminus \Az$: given $a \in A$ one may approximate it by $a' \in \Az$ and factor $x a' y = 1$, but the resulting witnesses need not remain bounded as $a' \to a$, and without such control the factorization does not pass to the limit.

We impose two conditions and show that together they suffice. The first is that $A$ have \emph{real rank zero}. This makes projections abundant: every nonzero hereditary $C^*$-subalgebra of $A$, and in particular every nonzero closed ideal, contains a nonzero projection \cite[Thm.~2.6]{BP}. The hypothesis is not restrictive in the direction of the conclusion, since purely infinite simple $C^*$-algebras have real rank zero \cite{Zhang}. Pure infiniteness of $A$ can then be tested on projections alone: a simple $C^*$-algebra is purely infinite precisely when every nonzero hereditary $C^*$-subalgebra contains an infinite projection \cite[Exercise~5.7]{RLL}. The required infinite projections can be produced from infinite idempotents of $\Az$, converting the problem into one about the existence of infinite idempotents. Idempotents are algebraic objects, and they are the natural currency in which a ring-theoretic hypothesis can be spent.

The second condition is that $\Az$ be closed under holomorphic functional calculus. Its role is to supply idempotents \emph{inside $\Az$}. Given a projection $q \in A$, density provides an approximant $a_0 \in \Az$, but $a_0$ is not an idempotent, and no algebraic manipulation of $a_0$ will make it one. What does produce an idempotent is the Riesz integral \cite[\S I.4]{Bla}: the spectrum of $a_0$ splits into two clusters near $0$ and $1$, and integrating the resolvent around the cluster near $1$ yields an idempotent $e$ with $\|e-q\|$ small. This integral converges in the norm of $A$, and there is no reason for its value to lie in $\Az$ --- unless $\Az$ is hfc-closed, which is exactly the assumption that it does. Thus hfc-closure is not a technical convenience: it is the precise hypothesis under which approximate analytic data can be converted into an exact algebraic certificate located in $\Az$, where the ring-theoretic hypothesis can act on it. Once $e \in \Az$ is available, the pure infiniteness of the ring $\Az$ makes $e$ an infinite idempotent \cite[Prop.~1.5]{AGP}, and standard bookkeeping \cite[Ch.~3]{RLL} transports this back to $q$, showing that every nonzero projection of $A$ is infinite.

A conditional theorem is worth little if its hypotheses can only be met trivially, and here the danger is real: a dense subalgebra that is closed in the norm of $A$ is all of $A$. We therefore exhibit proper examples. If a compact Lie group acts strongly continuously on a unital purely infinite simple $C^*$-algebra $A$, the algebra $A^\infty$ of elements with $C^\infty$ orbit map is dense (G\aa{}rding smoothing; see \cite{BC}), is a Fr\'echet algebra in the topology given by the derivatives of the action, and is inverse-closed --- because inversion is analytic on the invertibles of a Banach algebra, so that the orbit map of $a^{-1}$ is again smooth. By a theorem of Schweitzer \cite[Lem.~1.2]{Schw92}, a dense inverse-closed Fr\'echet subalgebra is hfc-closed. The concrete case is the gauge action of $\mathbb{T}$ on a Cuntz algebra \cite{Cuntz77, Raeburn}: $\mathcal{O}_n^\infty$ is a dense hfc-closed subalgebra of $\mathcal{O}_n$, purely infinite simple as a ring, and \emph{properly} contained in $\mathcal{O}_n$ --- the element $\sum_{d\ge1} d^{-2}s_1^{\,d}$ converges in norm but its orbit map is not three times differentiable. So the class of algebras to which our results apply is populated by proper subalgebras, and $\mathcal{O}_n^\infty$ is a purely infinite simple ring whose $C^*$-completion is purely infinite simple and strictly larger.

The algebraic theory of pure infiniteness begins with Ara, Goodearl and Pardo \cite{AGP}, who computed $K_0$ in the regular case, building on the exchange property established in \cite{AraExch}; it was extended to the non-simple setting by Aranda~Pino, Goodearl, Perera and Siles~Molina \cite{AGPSM}, who established permanence under matrices, corners, extensions and Morita equivalence and posed the problem quoted above. Surveys and systematic accounts are \cite{Abr, AAS}. On the analytic side the theory of purely infinite simple $C^*$-algebras goes back to Cuntz \cite{Cuntz77, Cuntz81}, with structural results of Lin--Zhang \cite{LinZhang} and R\o{}rdam \cite{Ror95}, criteria of Jolissaint--Robertson \cite{JR}, the non-simple theory of Kirchberg--R\o{}rdam \cite{KR}, and R\o{}rdam's example of a simple $C^*$-algebra with both a finite and an infinite projection \cite{Ror03}. The most substantial progress on Problem 8.4 known to us is that of Brown, Clark and an~Huef \cite{BCH}, who study it from the side of the reduced $C^*$-algebra $C^*_r(G)$ of an ample Hausdorff groupoid $G$. The natural dense subalgebra there, the Steinberg algebra, is too small for the converse to hold, so they introduce an intermediate $*$-algebra $B(G)$ assembled from the corners $1_U C^*_r(G) 1_U$; their main theorem is that, for $G$ minimal and effective, $B(G)$ is algebraically properly infinite if and only if $C^*_r(G)$ is purely infinite simple. This settles the transfer in a specific groupoid setting, by a route disjoint from ours: it builds an auxiliary algebra absorbing multiplication by elements of the completion, whereas our argument requires no such absorption but does require real rank zero.

Pure infiniteness has also been carried well beyond rings and $C^*$-algebras, with instructive consequences. Phillips introduced the notion for unital simple Banach algebras in his work on $L^p$-analogues of the Cuntz algebras \cite{Phillips}, and it agrees there with the ring-theoretic definition of \cite{AGP}. The agreement does not make the theories interchangeable: Daws and Horv\'ath \cite{DawsHorvath}, continuing their study of ring-theoretic finiteness in reduced products \cite{DawsHorvath2}, constructed a Cuntz-like Banach $*$-algebra that is purely infinite but has ultrapowers which are not even simple, in sharp contrast with the $C^*$-case. The notion has since been extended to non-unital Banach algebras and to Banach algebras of twisted \'etale groupoids \cite{KwasniewskiEtAl}. In a different direction, the holomorphic functional calculus on which our argument turns has itself been the subject of recent work in the Banach-algebraic setting: Hartmann and Lesch \cite{HartmannLesch} develop a multivariate holomorphic calculus for commuting tuples in a general Banach algebra, with explicit formulas from which the continuity properties can be read off directly; earlier, Lesch \cite{Lesch17} used divided differences to explain combinatorial identities arising in the spectral geometry of noncommutative tori \cite{ConnesMoscovici}.

Section~\ref{sec:setting} fixes conventions and recalls the facts about real rank zero used later. Section~\ref{sec:idem} contains the technical heart: it shows how a Riesz idempotent turns an element of $\Az$ near a projection of $A$ into an exact idempotent of $\Az$, and how algebraic infiniteness transfers between the two. Section~\ref{sec:hfc} isolates the role of holomorphic functional calculus and records the properties of smooth dense subrings. Section~\ref{sec:transfer} proves the transfer theorem, and Section~\ref{sec:converse} the converse, the resulting characterization, and the obstruction from finite projections. Section~\ref{sec:examples} carries out the example: the smooth subalgebra $\mathcal{O}_n^\infty$ of a Cuntz algebra under the gauge action, together with the closing discussion.

\section{Conventions}\label{sec:setting}

Throughout, $A$ is a unital $C^*$-algebra and $\Az \subseteq A$ is a \emph{dense unital subring}: a subring closed under addition and multiplication, containing a multiplicative identity $1_{\Az}$, and norm-dense in $A$. The following elementary observation is used silently throughout.

\begin{lemma}[density transfers the unit]\label{lem:unit}
If $\Az$ is a dense unital subring of $A$, then $1_{\Az} = 1_A$. In particular $A$ is automatically unital with the same identity.
\end{lemma}

\begin{proof}
Write $e := 1_{\Az}$. For $a \in A$ choose $a_n \in \Az$ with $a_n \to a$. Then $e a_n = a_n = a_n e$ for all $n$; letting $n \to \infty$ and using norm-continuity of multiplication gives $ea = a = ae$. Thus $e$ is a two-sided identity for $A$, so $A$ is unital and $1_A = e = 1_{\Az}$. (Density is essential: $\CC p$ for a proper projection $p$ is a unital subring of $A$ with $1 = p \ne 1_A$, but it is not dense.)
\end{proof}

\begin{definition}[ring-theoretic notions]\label{def:ring}
Let $R$ be a ring and $e,f \in R$ idempotents.
\begin{enumerate}
\item $f \ler e$ if $ef = fe = f$; and $f \lr e$ if additionally $f \ne e$.
\item $e \sr f$ if there exist $s,t \in R$ with $st = e$ and $ts = f$.
\item $e$ is an \emph{infinite idempotent} if $e \sr f \lr e$ for some idempotent $f$.
\item $R$ (unital) is \emph{purely infinite simple} if $R$ is simple and every nonzero right ideal contains an infinite idempotent \cite[Def.~1.2]{AGP}.
\end{enumerate}
\end{definition}

We shall use the equivalent characterization of Ara, Goodearl and Pardo \cite[Thm.~1.6]{AGP}: a unital ring $R$ is purely infinite simple if and only if it is not a division ring and, for every nonzero $a \in R$, there exist $x,y \in R$ with $xay = 1$.

\begin{definition}[spectra and inverse-closure]\label{def:invclosed}
For a unital algebra $R$ and $a \in R$ write $\Inv(R)$ for the group of invertible elements and
\[
\spec_R(a) = \{\lambda \in \CC : \lambda 1 - a \notin \Inv(R)\}
\]
for the spectrum computed in $R$. A unital subring $\Az \subseteq A$ is \emph{inverse-closed} (or \emph{spectrally permanent}) in $A$ if
\[
a \in \Az \cap \Inv(A) \implies a^{-1} \in \Az ,
\]
that is, an element of $\Az$ invertible in the larger algebra already has its inverse in $\Az$. Note the inverse is then unique and agrees with the one computed in $A$. Lemma~\ref{lem:implications}(2) records that this is equivalent to $\spec_{\Az}(a) = \spec_A(a)$ for all $a \in \Az$; the inclusion $\spec_A(a) \subseteq \spec_{\Az}(a)$ always holds, so the content is that $\Az$ sees no extra spectrum.
\end{definition}

\begin{definition}[hfc-closure]\label{def:hfc}
$\Az$ is \emph{closed under holomorphic functional calculus in $A$} (\emph{hfc-closed}, or \emph{smooth}) if for every $a \in \Az$ and every function $f$ holomorphic on an open neighborhood of $\spec_A(a)$, the element $f(a) \in A$ given by the Riesz--Dunford integral lies in $\Az$.
\end{definition}

\begin{remark}[hfc-closure is not norm-closure]\label{rem:notclosed}
The prototype is $C^\infty[0,1] \subsetneq C[0,1]$: dense, proper, not norm-closed, yet stable under holomorphic (and smooth) functional calculus.
\end{remark}

\begin{definition}[$C^*$-notions]\label{def:cstar}
Let $A$ be a $C^*$-algebra and $p,q \in A$ projections.
\begin{enumerate}
\item $p \sMvN q$ (\emph{Murray--von Neumann equivalence}) if $v^*v = p$ and $vv^* = q$ for some $v \in A$.
\item $p$ is an \emph{infinite projection} if $p \sMvN q$ for some projection $q \le p$ with $q \ne p$.
\item $A$ is \emph{purely infinite simple} if $A$ is simple, $A \ne \CC$, and every nonzero hereditary $C^*$-subalgebra contains an infinite projection.
\end{enumerate}
\end{definition}

Definition~\ref{def:cstar}(3) is from Cuntz's article \cite{Cuntz81}. Several equivalent formulations are known; in particular, for a unital $C^*$-algebra it is equivalent to requiring $A \ne \CC$ and that for every nonzero $x \in A$ there exist $\alpha, \beta \in A$ with $\alpha x \beta = 1$ \cite[Prop.~6.11.5]{BlaK}, the exact analogue of the ring-theoretic criterion above, and, equivalently, that for every nonzero positive $a \in A$ there exists $x \in A$ with $x^* a x = 1$ \cite[Exercise~5.7]{RLL}. It is the formulation in (3) that we verify.

\begin{definition}[real rank zero]\label{def:rr0}
A unital $C^*$-algebra $A$ has \emph{real rank zero} if the invertible self-adjoint elements are dense in $A_{sa}$ \cite{BP}. Equivalently, $A$ has property (FS): the self-adjoint elements of finite spectrum are dense in $A_{sa}$ \cite[Thm.~2.6]{BP}.
\end{definition}

\begin{lemma}[Brown--Pedersen]\label{fact:rr0}
If $A$ has real rank zero, then every nonzero hereditary $C^*$-subalgebra $B \subseteq A$ (in particular every nonzero closed two-sided ideal) contains a nonzero projection.
\end{lemma}

\begin{proof}
Since $A$ has real rank zero, every nonzero hereditary $C^*$-subalgebra $B$ has an approximate unit consisting of projections \cite[Thm.~2.6]{BP}; as $B \ne 0$, at least one of these projections is nonzero.
\end{proof}
\section{From approximate projections to infinite idempotents}\label{sec:idem}

This section assembles the elementary machinery used in the proof of the transfer theorem. Nothing here involves pure infiniteness or density; the results are statements about a single unital $C^*$-algebra $A$, and are collected in the order in which Theorem~\ref{thm:transfer} uses them.

The overall movement is in three steps. Given a projection $q$ of $A$ and an element $a$ of a subring close to it, we first \emph{capture} an idempotent $e$ near $q$ by holomorphic functional calculus, and observe that proximity forces $e$ and $q$ to be similar (\S\ref{ss:capture}). Since the ring-theoretic hypothesis will produce information about $e$ as an \emph{idempotent}, whereas pure infiniteness of $A$ is a statement about \emph{projections}, we next record how idempotents are converted to projections and how ring equivalence of projections upgrades to Murray--von Neumann equivalence (\S\ref{ss:convert}). Finally we combine these to transport infiniteness from an idempotent of a subring to a projection of $A$ (\S\ref{ss:transport}).

\subsection{Capturing an idempotent near a projection}\label{ss:capture}

If $a$ is close to a projection, its spectrum is confined to two small clusters, and the spectral idempotent belonging to the cluster near $1$ is close to the projection.

\begin{lemma}[spectral perturbation]\label{lem:spec}
Let $q \in A$ be a projection and $a \in A$ with $\|a - q\| \le \varepsilon$. Then
$\spec(a) \subseteq \{\lambda \in \CC : \dist(\lambda,\{0,1\}) \le \varepsilon\}$.
\end{lemma}

\begin{proof}
For $\lambda$ with $d := \dist(\lambda,\{0,1\}) > \varepsilon$: since $\spec(q) \subseteq \{0,1\}$, $\lambda - q$ is invertible with $\|(\lambda - q)^{-1}\| = 1/d$. Then
$\lambda - a = (\lambda - q)\bigl(1 - (\lambda - q)^{-1}(a - q)\bigr)$
and $\|(\lambda - q)^{-1}(a - q)\| \le \varepsilon/d < 1$, so $\lambda - a$ is invertible by the Neumann series.
\end{proof}

\begin{lemma}[Riesz idempotent with estimate]\label{lem:riesz}
Let $q \ne 0$ be a projection, $a \in A$, $\|a - q\| \le \varepsilon < 1/4$. Let $\Gamma = \{\lambda : |\lambda - 1| = 1/2\}$, positively oriented, and set
\[
e := \frac{1}{2\pi i}\oint_{\Gamma} (\lambda - a)^{-1}\, d\lambda .
\]
Then $e$ is a well-defined idempotent, $e = \chi(a)$ where $\chi$ is holomorphic on a neighborhood of $\spec(a)$ ($\chi \equiv 1$ near the part of the spectrum close to $1$, $\chi \equiv 0$ near the part close to $0$), and
\[
\|e - q\| \le \frac{2\varepsilon}{1 - 2\varepsilon} < 1 .
\]
In particular, if $a \in \Az$ and $\Az$ is hfc-closed, then $e \in \Az$.
\end{lemma}

\begin{proof}
\emph{Step 1: the contour avoids the spectrum.} By Lemma~\ref{lem:spec}, $\spec(a)$ lies in the $\varepsilon$-neighbourhood of $\{0,1\}$, i.e.\ in two disjoint disks, one about $0$ and one about $1$. Every $\lambda \in \Gamma$ satisfies $|\lambda - 1| = 1/2$ and, by the reverse triangle inequality, $|\lambda| \ge 1 - 1/2 = 1/2$; thus $\lambda$ is at distance $\ge 1/2$ from $\{0,1\}$, hence at distance $\ge 1/2 - \varepsilon > \varepsilon$ from $\spec(a)$. Since $\varepsilon < 1/4$, this gap is positive, so $\Gamma \cap \spec(a) = \emptyset$; as $\lambda \mapsto (\lambda - a)^{-1}$ is norm-continuous off the spectrum and $\Gamma$ is compact, the integral defining $e$ exists.

\emph{Step 2: $e$ is an idempotent.} The disk about $1$ lies inside $\Gamma$ and the disk about $0$ lies outside, so $\Gamma$ encircles exactly the part of $\spec(a)$ near $1$. Let $\chi$ be the function equal to $1$ on a neighbourhood of the near-$1$ disk and $0$ on a neighbourhood of the near-$0$ disk; being locally constant on a neighbourhood of the (disconnected) spectrum, $\chi$ is holomorphic there, and $e = \chi(a)$ by the holomorphic functional calculus. Since $\chi$ takes only the values $0$ and $1$ we have $\chi^2 = \chi$, and as the functional calculus is multiplicative, $e^2 = \chi(a)^2 = (\chi^2)(a) = \chi(a) = e$. Finally, if $a \in \Az$ and $\Az$ is hfc-closed then $e = \chi(a) \in \Az$.

\emph{Step 3: the estimate $\|e - q\| \le \frac{2\varepsilon}{1-2\varepsilon}$.} As $\spec(q) \subseteq \{0,1\}$ and only the point $1$ lies inside $\Gamma$, the same functional calculus gives $q = \frac{1}{2\pi i}\oint_{\Gamma} (\lambda - q)^{-1}\, d\lambda$. Subtracting the two contour integrals,
\[
e - q = \frac{1}{2\pi i}\oint_{\Gamma} \big[(\lambda - a)^{-1} - (\lambda - q)^{-1}\big]\, d\lambda .
\]
The integrand factors through the resolvent identity
\[
(\lambda - a)^{-1} - (\lambda - q)^{-1} = (\lambda - a)^{-1}(a - q)(\lambda - q)^{-1},
\]
and each of the three factors is bounded on $\Gamma$:
\begin{itemize}
\item $\|a - q\| \le \varepsilon$, by hypothesis;
\item $\|(\lambda - q)^{-1}\| \le 2$: the projection identity $(\lambda - q)^{-1} = \frac{1}{\lambda - 1}q + \frac{1}{\lambda}(1 - q)$ gives $\|(\lambda - q)^{-1}\| = \max\{|\lambda-1|^{-1}, |\lambda|^{-1}\}$, and on $\Gamma$ we have $|\lambda - 1| = 1/2$ and $|\lambda| \ge 1/2$;
\item $\|(\lambda - a)^{-1}\| \le \frac{2}{1 - 2\varepsilon}$: writing $\lambda - a = (\lambda - q)\big[1 - (\lambda-q)^{-1}(a - q)\big]$ and noting $\|(\lambda-q)^{-1}(a-q)\| \le 2\varepsilon < 1$, a Neumann series bounds the inverse bracket by $(1 - 2\varepsilon)^{-1}$, so $\|(\lambda - a)^{-1}\| \le 2(1-2\varepsilon)^{-1}$.
\end{itemize}
Hence $\|(\lambda - a)^{-1} - (\lambda - q)^{-1}\| \le \frac{2}{1-2\varepsilon}\cdot\varepsilon\cdot 2 = \frac{4\varepsilon}{1-2\varepsilon}$ on $\Gamma$. With $\mathrm{length}(\Gamma) = \pi$, the standard estimate $\big\|\frac{1}{2\pi i}\oint_\Gamma F\big\| \le \frac{1}{2\pi}(\max_\Gamma\|F\|)\,\mathrm{length}(\Gamma)$ gives
\[
\|e - q\| \le \frac{1}{2\pi}\cdot \pi \cdot \frac{4\varepsilon}{1 - 2\varepsilon} = \frac{2\varepsilon}{1 - 2\varepsilon},
\]
which is $< 1$ exactly when $\varepsilon < 1/4$.
\end{proof}

If an idempotent is close enough to a projection, the two are similar, and an explicit invertible implements the similarity.

\begin{lemma}[similarity of a close idempotent to a projection]\label{lem:sim}
Let $q$ be a projection and $e$ an idempotent with $\|e - q\| < 1$. Then $w := qe + (1-q)(1-e)$ is invertible in $A$ and $q = w e w^{-1}$.
\end{lemma}

\begin{proof}
Expanding, $w = 1 + 2qe - q - e$, while $(2q - 1)(e - q) = 2qe - 2q - e + q = 2qe - q - e$; hence $w = 1 + (2q - 1)(e - q)$. Since $2q - 1$ is a self-adjoint unitary, $\|w - 1\| \le \|e - q\| < 1$, so $w$ is invertible. Directly, $we = qe = qw$ (the terms $(1-q)(1-e)e$ and $q(1-q)(1-e)$ vanish), whence $e = w^{-1} q w$, i.e.\ $q = w e w^{-1}$.
\end{proof}

\subsection{From idempotents to projections}\label{ss:convert}

The ring-theoretic input of \S\ref{sec:transfer} concerns idempotents, while pure infiniteness of $A$ is formulated for projections. The next three lemmas supply the passage, and are classical; we include proofs to keep the paper self-contained.

\begin{lemma}[Kaplansky: from an idempotent to a projection]\label{lem:kap}
Let $g \in A$ be an idempotent. Put $z := 1 - (g - g^*)^2$. Then $z \ge 1$ is invertible, commutes with $g$ and $g^*$, and $r := g g^* z^{-1}$ is a projection satisfying
\[
gr = r, \qquad rg = g, \qquad rA = gA, \qquad g \sr r .
\]
\end{lemma}

\begin{proof}
$g - g^*$ is skew-adjoint, so $(g - g^*)^2 = -(g - g^*)(g - g^*)^* \le 0$ and $z \ge 1$; in particular $z$ is invertible and self-adjoint. Using $g^2 = g$ and $(g^*)^2 = g^*$:
\[
(g - g^*)^2 = g + g^* - g g^* - g^* g .
\]
Then $g(g - g^*)^2 = g + g g^* - g g^* - g g^* g = g - g g^* g$ and $(g - g^*)^2 g = g + g^* g - g g^* g - g^* g = g - g g^* g$; these agree, so $zg = gz$, and taking adjoints (using $z = z^*$) gives $z g^* = g^* z$. Hence $z^{-1}$ also commutes with $g$ and $g^*$.

\emph{$r^* = r$:} since $z^{-1}$ commutes with $g g^*$, $r^* = z^{-1} g g^* = g g^* z^{-1} = r$.

\emph{$r^2 = r$:} it suffices to show $g g^* g g^* = g g^* z$. Compute, using $g g^* g^* = g g^*$ and $g g^* g^* g = g g^* g$:
\[
g g^* z = g g^* - g g^*(g + g^* - g g^* - g^* g) = g g^* - g g^* g - g g^* + g g^* g g^* + g g^* g = g g^* g g^* .
\]
Hence $r^2 = g g^* z^{-1} g g^* z^{-1} = (g g^* g g^*) z^{-2} = g g^* z \cdot z^{-2} = r$.

\emph{$rg = g$:} $rg = g g^* g\, z^{-1}$, and from $g(g - g^*)^2 = g - g g^* g$ we get $gz = g - g(g - g^*)^2 = g g^* g$, so $rg = gz\, z^{-1} = g$.

\emph{$gr = r$:} $gr = g \cdot g g^* z^{-1} = g g^* z^{-1} = r$.

\emph{$rA = gA$:} $r = gr \in gA$ and $g = rg \in rA$.

\emph{$g \sr r$:} take $s = r$, $t = g$; then $st = rg = g$ and $ts = gr = r$.
\end{proof}

\begin{lemma}[transitivity of $\sr$ on idempotents]\label{lem:trans}
If $p_1 \sr p_2$ and $p_2 \sr p_3$ (all idempotents), then $p_1 \sr p_3$.
\end{lemma}

\begin{proof}
Say $st = p_1$, $ts = p_2$, $uv = p_2$, $vu = p_3$. Then
$(su)(vt) = s(uv)t = s p_2 t = s(ts)t = (st)^2 = p_1$
and
$(vt)(su) = v(ts)u = v p_2 u = v(uv)u = (vu)^2 = p_3$.
\end{proof}

\begin{lemma}[ring equivalence of projections upgrades to Murray--von Neumann equivalence]\label{lem:mvn}
Let $q, r \in A$ be projections with $q \sr r$. Then $q \sMvN r$.
\end{lemma}

\begin{proof}
Choose $s, t$ with $st = q$, $ts = r$, and set $s' := q s r$, $t' := r t q$. Then, using $q = st$ and $r = ts$:
\[
s' t' = q s (r t) q = q s (t s t) q = q (st)(st) q = q, \qquad
t' s' = r t (q s) r = r t (s t s) r = r (ts)(ts) r = r .
\]
Note $s' = q s' = s' r$ and $t' = r t' = t' q$. Put $x := s'^* s' \in rAr$, so $x \ge 0$. Then $x$ is invertible in the corner $rAr$ with inverse $t' t'^*$: indeed
\[
x\,(t' t'^*) = s'^* (s' t') t'^* = s'^* q\, t'^* = s'^* t'^* = (t' s')^* = r,
\]
and symmetrically $(t' t'^*)\, x = t' (s' t')^* s' = t' q s' = t' s' = r$. Let $x^{-1/2}$ be computed by functional calculus in the unital $C^*$-algebra $rAr$, and set $v := s' x^{-1/2}$. Then
\[
v^* v = x^{-1/2} x\, x^{-1/2} = r, \qquad
v v^* = s' x^{-1} s'^* = s' (t' t'^*) s'^* = (s' t')(s' t')^* = q q^* = q .
\]
So $q \sMvN r$.
\end{proof}

\subsection{Transporting infiniteness}\label{ss:transport}

Combining the two previous subsections gives the statement actually invoked in Step 3 of Theorem~\ref{thm:transfer}: an idempotent that is infinite in the ring sense yields a projection that is infinite in the $C^*$ sense.

\begin{lemma}[a projection that is an infinite idempotent is an infinite projection]\label{lem:transfer}
Let $q \in A$ be a projection which is infinite as an idempotent of the ring $A$. Then $q$ is an infinite projection of the $C^*$-algebra $A$.
\end{lemma}

\begin{proof}
The argument is the standard idempotent-to-projection bookkeeping (cf.\ \cite[Ch.~3]{RLL}), included for completeness. Choose an idempotent $g$ with $q \sr g \lr q$, i.e.\ $gq = qg = g$ and $g \ne q$. By Lemma~\ref{lem:kap} there is a projection $r$ with $gr = r$, $rg = g$, $rA = gA$, and $g \sr r$. By Lemma~\ref{lem:trans}, $q \sr r$, and by Lemma~\ref{lem:mvn}, $q \sMvN r$.

\emph{$r \le q$:} from $qg = g$ we get $gA \subseteq qA$, and $r = gr \in gA \subseteq qA$, so $qr = r$, i.e.\ $r \le q$.

\emph{$r \ne q$:} if $r = q$ then $qA = rA = gA$, so $q \in gA$, giving $gq = q$; but $gq = g$, forcing $g = q$, a contradiction.

Thus $q \sMvN r \le q$ with $r \ne q$: $q$ is an infinite projection.
\end{proof}

Finally, infiniteness is insensitive both to enlarging the ambient ring and to conjugation, which is what allows the property to be moved from $\Az$ to $A$ and then from $e$ to $q$.

\begin{lemma}[persistence and similarity-invariance of infiniteness]\label{lem:persist}
\leavevmode
\begin{enumerate}
\item If $e \in \Az$ is an infinite idempotent of the ring $\Az$, then $e$ is an infinite idempotent of the ring $A$.
\item If $e$ is an infinite idempotent of $A$ and $w \in A$ is invertible, then $w e w^{-1}$ is an infinite idempotent of $A$.
\end{enumerate}
\end{lemma}

\begin{proof}
(1) The elements $f, s, t \in \Az$ (with $st = e$, $ts = f$, $ef = fe = f$, $f \ne e$) satisfy the same relations in $A$, and $f \ne e$ persists. (2) $x \mapsto w x w^{-1}$ is a ring automorphism of $A$; apply it to $f$, $s$, $t$.
\end{proof}

\section{Smooth dense subrings}\label{sec:hfc}

We do not assume $\Az$ is self-adjoint or a priori an algebra over $\CC$; both are shown to follow from hfc-closure (Lemma~\ref{lem:auto}). This section collects what we need about hfc-closed subrings: two elementary consequences of the definition, a workable criterion for verifying hfc-closure, and two examples showing that the condition is restrictive. Pure infiniteness plays no role here.

\subsection{Elementary consequences}

\begin{lemma}[hfc-closed subrings are $\CC$-subalgebras]\label{lem:auto}
A nonzero hfc-closed subring $\Az \subseteq A$ contains $\CC 1_A$ and is closed under scalar multiplication. Hence it is a unital $\CC$-subalgebra of $A$, and the hypothesis ``subring'' in our results is no weaker than ``subalgebra''.
\end{lemma}

\begin{proof}
Fix $a \in \Az$. The constant function $f \equiv \lambda$ is holomorphic on a neighbourhood of $\spec_A(a)$ and $f(a) = \lambda 1_A$, so $\CC 1_A \subseteq \Az$; the function $g(z) = \lambda z$ gives $g(a) = \lambda a \in \Az$.
\end{proof}

\begin{lemma}\label{lem:implications}
Let $\Az$ be a dense unital subring of $A$.
\begin{enumerate}
\item If $\Az$ is hfc-closed then $\Az$ is inverse-closed.
\item $\Az$ is inverse-closed if and only if $\spec_{\Az}(a) = \spec_A(a)$ for every $a \in \Az$.
\end{enumerate}
\end{lemma}

\begin{proof}
(1) If $a \in \Az$ is invertible in $A$ then $0 \notin \spec_A(a)$, so $f(z) = z^{-1}$ is holomorphic on a neighbourhood of $\spec_A(a)$ and $f(a) = a^{-1} \in \Az$.

(2) The inclusion $\spec_A(a) \subseteq \spec_{\Az}(a)$ always holds, since an inverse lying in $\Az$ lies in $A$. If $\Az$ is inverse-closed and $\lambda \notin \spec_A(a)$, then $\lambda 1 - a$ is an element of $\Az$ invertible in $A$, so its inverse lies in $\Az$ and $\lambda \notin \spec_{\Az}(a)$; this gives the reverse inclusion. Conversely, if the spectra agree and $a \in \Az \cap \Inv(A)$, then $0 \notin \spec_A(a) = \spec_{\Az}(a)$, so $a^{-1} \in \Az$.
\end{proof}

The converse of \ref{lem:implications}(1) fails for general subrings: the Riesz--Dunford integral converges in the norm of $A$, and without completeness of $\Az$ in a suitable finer topology its value may escape $\Az$. Completeness is exactly what repairs this.

\subsection{The Fr\'echet criterion}

\begin{definition}[Fr\'echet subalgebra]\label{def:frechet}
A \emph{Fr\'echet algebra} is a complex algebra $B$ equipped with a topology that is
\begin{enumerate}
\item locally convex and Hausdorff, generated by a \emph{countable} family of seminorms $\|\cdot\|_0, \|\cdot\|_1, \dots$ (equivalently: metrizable, by $d(x,y) = \sum_k 2^{-k}\|x-y\|_k(1+\|x-y\|_k)^{-1}$);
\item complete;
\end{enumerate}
and for which multiplication is jointly continuous, i.e.\ for each $k$ there are $C_k > 0$ and $m_k$ with $\|xy\|_k \le C_k \sum_{i,j \le m_k}\|x\|_i\|y\|_j$ \cite[Def.~1.1]{Schw92}.

If $A$ is a Banach algebra, a \emph{Fr\'echet subalgebra} of $A$ is a subalgebra $\Az \subseteq A$ carrying a Fr\'echet algebra topology $\tau$ that is \emph{finer} than the norm topology of $A$; equivalently, the inclusion $(\Az,\tau) \hookrightarrow A$ is continuous \cite[Def.~1.1]{Schw92}.
\end{definition}

\begin{remark}\label{rem:why-countable}
Countability in (1) is not a formality: it is equivalent to metrizability, and metrizability is what the proof of Theorem~\ref{thm:schweitzer} consumes, through facts (P) and (O) below --- both of which are statements about completely metrizable spaces and are false for general topological groups.
\end{remark}

\begin{theorem}[Schweitzer]\label{thm:schweitzer}
Let $A$ be a unital Banach algebra and $\Az \subseteq A$ a dense unital Fr\'echet subalgebra in the sense of Definition~\ref{def:frechet}. Then
\[
\Az \text{ is inverse-closed in } A \iff \Az \text{ is hfc-closed in } A .
\]
\end{theorem}

The implication $(\Leftarrow)$ is Lemma~\ref{lem:implications}(1) and needs no hypothesis on $\tau$. For $(\Rightarrow)$, which is \cite[Lem.~1.2]{Schw92}, we record the argument in the streamlined form of \cite{Roe}; it uses two standard facts:
\begin{itemize}
\item[(P)] a group with a completely metrizable topology and jointly continuous multiplication has continuous inversion \cite{Pfister};
\item[(O)] an open subset of a completely metrizable space is completely metrizable \cite[Ch.~24]{Willard}.
\end{itemize}

\begin{proof}[Proof of $(\Rightarrow)$]
\emph{Step 1: inversion is $\tau$-continuous.} Inverse-closedness gives $\Inv(\Az) = \Az \cap \Inv(A)$. As $\Inv(A)$ is open in $A$ and the inclusion $\Az \hookrightarrow A$ is $\tau$-continuous, $\Inv(\Az)$ is $\tau$-open in $\Az$; by (O) it is completely metrizable. Multiplication is jointly $\tau$-continuous, so (P) applies: $g \mapsto g^{-1}$ is $\tau$-continuous on $\Inv(\Az)$.

\emph{Step 2: the resolvent is $\tau$-continuous.} Let $a \in \Az$ and let $\Gamma$ be a contour enclosing $\spec_A(a)$ inside the domain of $f$. For $\lambda \in \Gamma$ the element $\lambda 1 - a$ lies in $\Az$ and is invertible in $A$, hence in $\Az$. The map $\lambda \mapsto \lambda 1 - a$ is $\tau$-continuous, and composing with Step 1 shows $\lambda \mapsto (\lambda 1 - a)^{-1}$ is continuous from $\Gamma$ into $(\Az,\tau)$.

\emph{Step 3: the integral converges in $\Az$.} The integrand $\lambda \mapsto f(\lambda)(\lambda 1 - a)^{-1}$ is $\tau$-continuous on the compact set $\Gamma$, hence uniformly so; its Riemann sums form a $\tau$-Cauchy net in $\Az$, which converges since $\tau$ is complete. The limit is $f(a)$, so $f(a) \in \Az$.
\end{proof}

\begin{remark}[where each hypothesis acts]\label{rem:roles}
Inverse-closedness identifies $\Inv(\Az)$ with $\Az \cap \Inv(A)$ and so makes it open; completeness of $\tau$ powers (P) in Step~1 and closes the integral in Step~3; continuity of the inclusion lets $\Az$ inherit the openness of $\Inv(A)$. Removing completeness breaks Step~1, which is precisely why the converse of Lemma~\ref{lem:implications}(1) fails in general.
\end{remark}

\subsection{Two non-examples}

The next two examples show that neither completeness alone nor density alone gives inverse-closedness, and that the combinatorial dense subalgebras of most interest in the algebraic theory fail it.

\begin{example}[completeness is not enough]\label{ex:free}
Let $F_2$ be the free group on generators $a,b$ and consider $\ell^1(F_2)$ inside the reduced group $C^*$-algebra $C^*_r(F_2)$. This is a dense Banach $*$-subalgebra, so it is complete in its own norm. Put
\[
h = a + a^{-1} + b + b^{-1} .
\]
Kesten's theorem \cite[Thm.~3]{Kesten} (see also \cite{AkemannOstrand}) computes the norm of the sum of the generators and their inverses under the left-regular representation of a free group: for the free group on $n$ generators this norm is $2\sqrt{2n-1}$. For $n = 2$ this gives $\|h\|_{C^*_r(F_2)} = 2\sqrt3$, which is also the spectral radius since $h$ is self-adjoint. In $\ell^1(F_2)$ all coefficients of $h^n$ are nonnegative, so $\|h^n\|_1 = 4^n$ and the spectral radius is $4$. The two spectral radii differ, hence so do the two spectra, and $\ell^1(F_2)$ is not inverse-closed by Lemma~\ref{lem:implications}(2).
\end{example}

\begin{example}[Leavitt and Steinberg algebras]\label{ex:laurent}
Let $\CC[z,z^{-1}]$ be the Laurent polynomials, dense in $C(\mathbb{T})$ by Stone--Weierstrass, and put
\[
h(z) = 3 - z - z^{-1}, \qquad h(e^{i\theta}) = 3 - 2\cos\theta \ge 1 .
\]
Then $h$ is invertible in $C(\mathbb{T})$, being bounded away from $0$. Its inverse has Fourier coefficients
\[
\widehat{h^{-1}}(n) = \tfrac{1}{\sqrt5}\,\rho^{\,|n|}, \qquad \rho = \tfrac{3-\sqrt5}{2} \in (0,1),
\]
all nonzero, whereas a Laurent polynomial has finite Fourier support. So $h^{-1} \notin \CC[z,z^{-1}]$ and the subalgebra is not inverse-closed.

Since $\CC[z,z^{-1}]$ is the Leavitt path algebra of the graph with one vertex and one loop, sitting inside its graph $C^*$-algebra $C(\mathbb{T})$, this is the smallest instance of a general phenomenon: Leavitt path algebras and Steinberg algebras are typically not inverse-closed in their completions, and therefore lie outside the scope of Theorem~\ref{thm:transfer}. We return to this in Remark~\ref{rem:regimes}.
\end{example}

\section{The transfer theorem}\label{sec:transfer}

We first record the ring-theoretic input.

\begin{proposition}[Ara--Goodearl--Pardo {\cite[Prop.~1.5]{AGP}}]\label{fact:agp}
In a purely infinite simple unital ring, every nonzero idempotent is infinite.
\end{proposition}

\noindent This is the sole ring-theoretic input to Theorem~\ref{thm:transfer}; it is used in the same form in \cite[Cor.~3.1]{BCH}.

\begin{theorem}[transfer: smooth algebraic pure infiniteness $\Rightarrow$ $C^*$ pure infiniteness]\label{thm:transfer}
Let $A$ be a unital $C^*$-algebra of real rank zero and $\Az \subseteq A$ a dense, hfc-closed unital subring that is purely infinite simple as a ring. Then $A$ is a purely infinite simple $C^*$-algebra.
\end{theorem}

\begin{proof}
Fix $\varepsilon = 1/8$; Lemma~\ref{lem:riesz} then yields $\|e - q\| \le \tfrac13 < 1$ for the Riesz idempotent of any $a$ with $\|a - q\| \le \varepsilon$.

\emph{Step 1 (idempotent capture).} Let $q \ne 0$ be a projection of $A$. By density pick $a_0 \in \Az$ with $\|a_0 - q\| \le 1/8$. By Lemma~\ref{lem:riesz} the Riesz idempotent $e = \chi(a_0)$ lies in $\Az$ and satisfies $\|e - q\| \le 1/3$; by Lemma~\ref{lem:sim} there is $w \in \Inv(A)$ with $q = w e w^{-1}$, and $e \ne 0$.

\emph{Step 2 (simplicity).} Let $I \ne 0$ be a closed two-sided ideal. By Lemma~\ref{fact:rr0}, $I$ contains a nonzero projection $q$; by Step~1, $e = w^{-1} q w \in I \cap \Az$ is nonzero. Then $I \cap \Az$ is a nonzero two-sided ideal of the simple ring $\Az$, so $\Az \subseteq I$, whence $I = A$.

\emph{Step 3 (infinite projections everywhere).} Let $B$ be a nonzero hereditary $C^*$-subalgebra; by Lemma~\ref{fact:rr0} it contains a nonzero projection $q$. By Step~1, $q = w e w^{-1}$ with $0 \ne e \in \Az$; by Proposition~\ref{fact:agp}, $e$ is an infinite idempotent of $\Az$; by Lemma~\ref{lem:persist} it is one in $A$ and so is $q$; by Lemma~\ref{lem:transfer}, $q$ is an infinite projection, and $q \in B$.

By Steps 2--3 and $A \ne \CC$ (it has an infinite projection), $A$ is purely infinite simple.
\end{proof}

\begin{remark}[what powers the proof]\label{rem:mechanism}
Self-adjointness of $\Az$ and any absorption property $\Az x \Az \subseteq \Az$ are unnecessary. Riesz idempotents furnish the \emph{exact} algebraic certificates that norm-controlled factorizations would otherwise have to supply; real rank zero guarantees projections suffice to detect pure infiniteness. Proposition~\ref{fact:agp} is purely ring-theoretic and Lemmas~\ref{lem:spec}--\ref{lem:persist} are elementary, so no comparison property of $A$ is presupposed, so the argument is not circular as an implication.
\end{remark}

\section{The converse, the characterization, and an obstruction}\label{sec:converse}

\subsection{The converse}\label{ss:converse}

The converse needs only inverse-closedness (Lemma~\ref{lem:implications}(1)), and no real rank hypothesis.

\begin{theorem}[converse]\label{thm:converse}
Let $A$ be a unital purely infinite simple $C^*$-algebra and $\Az \subseteq A$ a dense inverse-closed unital subring. Then $\Az$ is purely infinite simple as a ring.
\end{theorem}

\begin{proof}
We verify the criterion \cite[Thm.~1.6]{AGP}: a unital ring is purely infinite simple iff it is not a division ring and every nonzero element $a$ admits $x,y$ with $xay = 1$.

\emph{Factorization.} Let $0 \ne a \in \Az$. Since $A$ is purely infinite simple unital, there are $x, y \in A$ with $xay = 1$ \cite[Prop.~6.11.5]{BlaK}. Pick $x_0, y_0 \in \Az$ with $\|x_0 a y_0 - xay\| < 1$; then $u := x_0 a y_0 \in \Az$ has $\|u - 1\| < 1$, so $u \in \Inv(A)$ and, by inverse-closedness, $u^{-1} \in \Az$. Hence $(u^{-1} x_0)\,a\,y_0 = 1$ with both factors in $\Az$.

\emph{Not a division ring.} As $1$ is an infinite projection there is a nonunitary isometry $s$ ($s^*s = 1$, $ss^* \ne 1$). If $\|t - s\| < 1$ then $\|s^*t - 1\| < 1$, so $s^*t \in \Inv(A)$; were $t$ invertible, $s^* = (s^*t)t^{-1}$ would be too, forcing $s$ unitary and $ss^* = 1$, a contradiction. So the unit ball about $s$ misses $\Inv(A)$; a point $a_0 \in \Az$ with $\|a_0 - s\| < \tfrac12$ is nonzero and noninvertible in $A$, hence in $\Az$.

By \cite[Thm.~1.6]{AGP}, $\Az$ is purely infinite simple.
\end{proof}

\subsection{The characterization}\label{ss:char}

\begin{theorem}[characterization]\label{thm:char}
For a unital $C^*$-algebra $A$ of real rank zero and a dense hfc-closed unital subring $\Az \subseteq A$, the following are equivalent:
\begin{enumerate}
\item $A$ is purely infinite simple as a $C^*$-algebra;
\item $\Az$ is purely infinite simple as a ring.
\end{enumerate}
Moreover, if $A$ is purely infinite simple then it automatically has real rank zero \cite{Zhang}, and \emph{every} dense inverse-closed unital subring of $A$ is purely infinite simple as a ring.
\end{theorem}

\begin{proof}
(2)$\Rightarrow$(1) is Theorem~\ref{thm:transfer}; (1)$\Rightarrow$(2) is Theorem~\ref{thm:converse} (hfc-closed $\Rightarrow$ inverse-closed). The final clause is Zhang's theorem together with Theorem~\ref{thm:converse}.
\end{proof}

\subsection{An obstruction}\label{ss:obstruction}

\begin{theorem}[obstruction]\label{thm:obstruction}
Let $\Az \subseteq A$ be a dense hfc-closed unital subring that is purely infinite simple as a ring. Then every nonzero projection of $A$ is infinite.
\end{theorem}

\begin{proof}
For a projection $q \ne 0$, run Steps 1 and 3 of Theorem~\ref{thm:transfer}: real rank zero was used there only to \emph{produce} projections, whereas $q$ is now given. Thus $q$ is infinite.
\end{proof}

\begin{corollary}\label{cor:nofinite}
If a unital $C^*$-algebra $A$ contains a nonzero finite projection, then $A$ has no dense hfc-closed purely infinite simple unital subring.
\end{corollary}

\begin{proof}
Contrapositive of Theorem~\ref{thm:obstruction}.
\end{proof}

\begin{remark}[traces]\label{rem:trace}
If the unit is an algebraically properly infinite idempotent in some dense unital subring, then $A$ is traceless (independently of hfc-closure). Indeed $1 \oplus 1 \sr q' \ler 1 \oplus 0$ in $M_2(A)$, and a tracial state $\tau$ gives $2 = (\tau \otimes \mathrm{Tr})(1 \oplus 1) = (\tau \otimes \mathrm{Tr})(q') \le 1$ via similarity-invariance, a contradiction. In particular, if a counterexample to \cite[Problem~8.4]{AGPSM} exists, its $C^*$-algebra must be traceless.
\end{remark}

\begin{remark}[R\o{}rdam's example]\label{rem:rordam}
Corollary~\ref{cor:nofinite} applies in particular to R\o{}rdam's simple nuclear unital $C^*$-algebra containing both a finite and an infinite projection \cite{Ror03}: it has no dense hfc-closed purely infinite simple unital subring.
\end{remark}

\section{Examples: smooth subalgebras}\label{sec:examples}

Theorem~\ref{thm:transfer} assumes the existence of a dense hfc-closed subring that is purely infinite simple as a ring. If the only such subring were $\Az = A$ itself, the theorem would say nothing new: it would reduce to the known implication that a $C^*$-algebra purely infinite as a ring is purely infinite as a $C^*$-algebra \cite[Prop.~3.17]{AGPSM}. So we must exhibit \emph{proper} dense subrings satisfying the hypotheses.

The two obvious candidates fail or are unavailable. Norm-closed subalgebras are useless: a dense norm-closed subalgebra equals $A$ (Remark~\ref{rem:notclosed}). Combinatorial subalgebras such as Leavitt path algebras are dense and purely infinite simple as rings, but they are not inverse-closed, hence not hfc-closed (Example~\ref{ex:laurent}). What remains, and what the noncommutative-geometry literature supplies, are the \emph{smooth} subalgebras: subalgebras of ``differentiable'' elements for a group action. We now describe these concretely in the motivating case.

\subsection{The gauge action on a Cuntz algebra}

Recall that the Cuntz algebra $\mathcal{O}_n$ ($2 \le n < \infty$) \cite{Cuntz77} is the universal $C^*$-algebra generated by isometries $s_1,\dots,s_n$ subject to
\[
s_i^* s_j = \delta_{ij}\,1, \qquad \sum_{i=1}^n s_i s_i^* = 1 .
\]
For a word $\mu = (\mu_1,\dots,\mu_k)$ in $\{1,\dots,n\}$ write $s_\mu = s_{\mu_1}\cdots s_{\mu_k}$ and $|\mu| = k$, with $s_\emptyset = 1$. Products collapse by the rule
\[
s_\mu^* s_\nu = \begin{cases} s_\gamma & \text{if } \nu = \mu\gamma,\\ s_\gamma^* & \text{if } \mu = \nu\gamma,\\ 0 & \text{if } \mu,\nu \text{ are incomparable,}\end{cases}
\]
so the monomials $s_\mu s_\nu^*$ are closed under multiplication and adjoints. Their finite complex linear span
\[
L_n := \operatorname{span}_{\CC}\{\, s_\mu s_\nu^* \ :\ \mu,\nu \text{ words of arbitrary finite length}\,\}
\]
is therefore a unital $*$-subalgebra of $\mathcal{O}_n$, dense by construction; it is the \emph{Leavitt algebra}, the purely algebraic object presented by the same relations over $\CC$. It is a standard example of a ring that is purely infinite simple in the algebraic sense, and $\mathcal{O}_n = \overline{L_n}$. Leavitt path algebras and, more generally, Steinberg algebras are the analogous combinatorial dense subalgebras attached to graphs and to ample groupoids; $L_n$ is the case of the graph with one vertex and $n$ loops.

Fix $z$ in the unit circle $\mathbb{T}$. The assignment $s_i \mapsto z s_i$ preserves both defining relations, since $(zs_i)^*(zs_j) = |z|^2 s_i^*s_j = \delta_{ij}$ and $\sum_i (zs_i)(zs_i)^* = |z|^2 \cdot 1 = 1$. By universality it extends to an automorphism $\alpha_z$ of $\mathcal{O}_n$, and $z \mapsto \alpha_z$ is a strongly continuous action of $\mathbb{T}$, the \emph{gauge action}. (Cuntz \cite{Cuntz77} works instead with the fixed-point subalgebra of this action, the UHF core of type $n^\infty$, which he uses as the algebra of ``Fourier coefficients''; the circle action itself is presented in this form in the Cuntz--Krieger and graph-algebra literature, e.g.\ \cite{Raeburn}.) On monomials it acts by a character:
\begin{equation}\label{eq:gauge}
\alpha_z(s_\mu s_\nu^*) = z^{|\mu|}\,\overline{z}^{\,|\nu|}\, s_\mu s_\nu^* = z^{\,|\mu|-|\nu|}\, s_\mu s_\nu^* .
\end{equation}
Only the action is needed below; we use \eqref{eq:gauge} solely to compute with concrete elements.

\begin{lemma}[strong continuity]\label{lem:strongcont}
For each $a \in \mathcal{O}_n$ the orbit map $z \mapsto \alpha_z(a)$ is norm-continuous on $\mathbb{T}$.
\end{lemma}

\begin{proof}
On a monomial, \eqref{eq:gauge} gives $\|\alpha_z(s_\mu s_\nu^*) - \alpha_w(s_\mu s_\nu^*)\| = |z^{d} - w^{d}|$ with $d = |\mu|-|\nu|$, which tends to $0$ as $z \to w$; hence the orbit map is continuous for every $a \in L_n$, being a finite sum of such terms. For general $a$, each $\alpha_z$ is a $*$-automorphism of a $C^*$-algebra and therefore isometric, so the family $\{\alpha_z\}_{z\in\mathbb{T}}$ is uniformly bounded by $1$. Given $\varepsilon > 0$ choose $b \in L_n$ with $\|a - b\| < \varepsilon/3$; then
\[
\|\alpha_z(a) - \alpha_w(a)\| \le \|\alpha_z(a-b)\| + \|\alpha_z(b) - \alpha_w(b)\| + \|\alpha_w(b-a)\| < \tfrac{2\varepsilon}{3} + \|\alpha_z(b)-\alpha_w(b)\|,
\]
and the middle term is below $\varepsilon/3$ for $z$ near $w$.
\end{proof}

\begin{remark}
Strong continuity is continuity of $z \mapsto \alpha_z(a)$ for each fixed $a$; the stronger requirement that $z \mapsto \alpha_z$ be continuous in the norm of $\operatorname{Aut}(\mathcal{O}_n)$ fails, and is not what is meant.
\end{remark}

\subsection{Smooth subalgebras and their derivations}

Throughout this subsection $G$ is a compact Lie group acting strongly continuously on a unital $C^*$-algebra $A$ by $\alpha$. Write $\Phi_a : G \to A$, $\Phi_a(g) = \alpha_g(a)$, for the \emph{orbit map} of $a$. We use repeatedly that each $\alpha_g$, being a $*$-automorphism of a $C^*$-algebra, is isometric:
\begin{equation}\label{eq:isom}
\|\alpha_g(a)\| = \|a\| \qquad (g \in G,\ a \in A),
\end{equation}
which holds because a $*$-isomorphism preserves spectra, hence spectral radii, and $\|x\|^2 = \|x^*x\| = r(x^*x)$.

\begin{definition}[smooth elements]\label{def:smooth}
$A^\infty := \{a \in A : \Phi_a \in C^\infty(G,A)\}$, the \emph{smooth subalgebra} of the action.
\end{definition}

We first recall the Lie-theoretic notation used below; see \cite{Knapp} for background. A compact Lie group $G$ is in particular a smooth manifold, and its \emph{Lie algebra} $\g = \operatorname{Lie}(G) := T_eG$ is the tangent space at the identity --- a real vector space with $\dim_{\mathbb{R}}\g = \dim G =: m$. Each $X \in \g$ determines a unique one-parameter subgroup $t \mapsto \exp(tX)$, the integral curve through $e$ with velocity $X$ at $t = 0$; the resulting \emph{exponential map} $\exp : \g \to G$ restricts to a diffeomorphism from a neighbourhood of $0 \in \g$ onto a neighbourhood of $e \in G$, and is surjective when $G$ is compact and connected. Although $X$ lies in $\g$ and not in $G$, it enters the constructions below only through $\exp$: the vector $X$ selects a direction, $\exp$ converts it into the genuine group elements $\exp(tX) \in G$, and the action $\alpha$ is applied to those, never to $X$ itself. Composing the action with this one-parameter subgroup yields, for each $X \in \g$, the strongly continuous one-parameter group of isometries $t \mapsto \alpha_{\exp(tX)}$ on $A$, whose generator is the derivation $\delta_X$ defined next; that it is a group of \emph{isometries} is \eqref{eq:isom}, and this is what will make $\delta_X$ a closed operator in Theorem~\ref{thm:frechetstructure}. The passage to $\g$ is what secures countability: the orbit map records smoothness over the whole manifold $G$, which carries no natural countable family of seminorms, whereas a basis $X_1,\dots,X_m$ of the finite-dimensional space $\g$ produces the countably many iterated derivations $\delta_{j_1}\cdots\delta_{j_k}$, and hence the countable seminorm family under which $A^\infty$ is a Fr\'echet algebra.

Fix a basis $X_1,\dots,X_m$ of $\g = \operatorname{Lie}(G)$, so $m = \dim G$. For $X \in \g$ put
\[
\operatorname{dom}(\delta_X) := \Big\{a \in A \ :\ \lim_{t\to0}\tfrac{1}{t}\big(\alpha_{\exp(tX)}(a) - a\big) \ \text{exists in the norm of } A\Big\},
\]
and for $a \in \operatorname{dom}(\delta_X)$ let $\delta_X(a)$ denote that limit; write $\delta_j := \delta_{X_j}$. The map $X \mapsto \delta_X$ is linear.

\begin{lemma}[differentiability along the whole orbit]\label{lem:C1}
Let $X \in \g$, write $\alpha_t := \alpha_{\exp(tX)}$, and let $a \in \operatorname{dom}(\delta_X)$. Then $t \mapsto \alpha_t(a)$ is $C^1$ with $\frac{d}{dt}\alpha_t(a) = \alpha_t(\delta_X(a))$ for every $t$, and consequently
\[
\alpha_s(a) - a = \int_0^s \alpha_t\big(\delta_X(a)\big)\,dt \qquad (s \in \mathbb{R}).
\]
\end{lemma}

\begin{proof}
The definition gives differentiability only at $t=0$; the group law upgrades it. Since $\alpha_{t+h} = \alpha_t\alpha_h$ and $\alpha_t$ is linear,
\[
\frac{\alpha_{t+h}(a)-\alpha_t(a)}{h} = \alpha_t\!\left(\frac{\alpha_h(a)-a}{h}\right).
\]
As $h \to 0$ the bracket tends in norm to $\delta_X(a)$, and $\alpha_t$ is bounded by \eqref{eq:isom}, hence continuous, so it may be interchanged with the limit. The derivative $t \mapsto \alpha_t(\delta_X(a))$ is continuous by strong continuity, so the map is $C^1$; the integral formula is the fundamental theorem of calculus for $C^1$ maps into a Banach space.
\end{proof}

\begin{lemma}\label{lem:der}
Each $\delta_X$ is a densely defined closed derivation, and
\begin{equation}\label{eq:domdesc}
A^\infty = \bigcap_{k \ge 0}\ \bigcap_{j_1,\dots,j_k}\operatorname{dom}(\delta_{j_1}\cdots\delta_{j_k}).
\end{equation}
\end{lemma}

\begin{proof}
\emph{Derivation.} For $a,b \in \operatorname{dom}(\delta_X)$, multiplicativity of $\alpha_t$ gives
\[
\frac{\alpha_t(ab)-ab}{t} = \frac{\alpha_t(a)-a}{t}\,\alpha_t(b) \;+\; a\,\frac{\alpha_t(b)-b}{t},
\]
and $\alpha_t(b) \to b$ by strong continuity, so the limit exists: $ab \in \operatorname{dom}(\delta_X)$ with $\delta_X(ab) = \delta_X(a)b + a\delta_X(b)$.

\emph{Dense domain.} If $a \in A^\infty$ then $\Phi_a$ is $C^\infty$, so in particular $t \mapsto \alpha_{\exp(tX)}(a)$ is differentiable at $0$, which is exactly the condition for $a \in \operatorname{dom}(\delta_X)$. Thus $A^\infty \subseteq \operatorname{dom}(\delta_X)$, and $A^\infty$ is dense by Theorem~\ref{thm:dense} below.

\emph{Closedness.} Let $a_n \in \operatorname{dom}(\delta_X)$ with $a_n \to a$ and $\delta_X(a_n) \to y$; we must produce the limit defining $\delta_X(a)$ and identify it as $y$. By Lemma~\ref{lem:C1},
$\alpha_s(a_n) - a_n = \int_0^s \alpha_t(\delta_X(a_n))\,dt$,
an identity for each $n$ with no limit in $s$. The integrands converge uniformly in $t$, since $\|\alpha_t(\delta_X(a_n)) - \alpha_t(y)\| = \|\delta_X(a_n)-y\|$ by \eqref{eq:isom}, a bound independent of $t$; so we may pass to the limit and obtain
$\alpha_s(a) - a = \int_0^s \alpha_t(y)\,dt$.
Dividing by $s \ne 0$ and writing $y = \frac1s\int_0^s y\,dt$,
\[
\Big\|\frac{\alpha_s(a)-a}{s} - y\Big\| = \Big\|\frac1s\int_0^s\big[\alpha_t(y)-y\big]dt\Big\| \le \sup_{|t|\le|s|}\|\alpha_t(y)-y\| \xrightarrow[s\to0]{} 0
\]
by strong continuity at the identity. So $a \in \operatorname{dom}(\delta_X)$ with $\delta_X(a) = y$.

\emph{Proof of \eqref{eq:domdesc}.} $(\subseteq)$ The composite $\delta_{j_1}\cdots\delta_{j_k}$ is defined at $a$ when each stage lands in the next domain, so it suffices to know that $A^\infty$ is invariant under every $\delta_X$. Pulling the bounded map $\alpha_g$ through the defining limit,
\[
\Phi_{\delta_X(a)}(g) = \alpha_g\big(\delta_X(a)\big) = \frac{d}{dt}\Big|_{0}\alpha_{g\exp(tX)}(a) = \frac{d}{dt}\Big|_{0}\Phi_a\big(g\exp(tX)\big),
\]
so $\Phi_{\delta_X(a)}$ is a directional derivative of $\Phi_a$; a derivative of a $C^\infty$ map is $C^\infty$, whence $\delta_X(a) \in A^\infty$, and the chain never breaks.

$(\supseteq)$ Suppose all the iterated derivatives exist at $a$. For $u = \exp(t_1X_{j_1})\cdots\exp(t_kX_{j_k})$ we have $\Phi_a(gu) = \alpha_g(\alpha_u(a))$, and differentiating at the origin gives $\alpha_g(\delta_{j_1}\cdots\delta_{j_k}(a))$. In exponential coordinates of the second kind these are precisely the partial derivatives of $\Phi_a$; they exist in all orders, and each is continuous in $g$, being of the form $g \mapsto \alpha_g(y)$ with $y$ fixed. A map all of whose partial derivatives of all orders exist and are continuous is $C^\infty$, so $a \in A^\infty$.
\end{proof}

\begin{remark}[the number of derivations]\label{rem:onegen}
Since $X \mapsto \delta_X$ is linear, the derivations generated by the action span a space of dimension $\dim\g = \dim G$. For $G = \mathbb{T}$ this is $1$: every one-parameter subgroup is $t \mapsto e^{ict}$, and differentiating along it gives $c\,\delta$ where $\delta(a) = \frac{d}{d\theta}\big|_{0}\alpha_{e^{i\theta}}(a)$; then \eqref{eq:domdesc} reads $\mathcal{O}_n^\infty = \bigcap_{k\ge0}\operatorname{dom}(\delta^k)$. This is a statement about $\mathbb{T}$, not about $\mathcal{O}_n$, which has many other derivations --- for instance the inner ones $x \mapsto bx-xb$ with $b \notin \CC 1$, nonzero since $\mathcal{O}_n$ is simple with trivial centre.
\end{remark}

On a monomial \eqref{eq:gauge} gives $\alpha_{e^{i\theta}}(s_\mu s_\nu^*) = e^{i\theta d}s_\mu s_\nu^*$ with $d = |\mu|-|\nu|$, so the orbit map is a scalar function times a fixed element, hence $C^\infty$, and
\[
\delta(s_\mu s_\nu^*) = i\,d\;s_\mu s_\nu^*.
\]
In particular $L_n \subseteq \mathcal{O}_n^\infty$. Since $\delta$ multiplies a monomial of degree $d$ by $id$, it is unbounded, so by the closed graph theorem its domain is a proper subspace; the next lemma records the corresponding statement for $\mathcal{O}_n^\infty$.

\begin{lemma}[$\mathcal{O}_n^\infty$ is proper]\label{lem:proper}
$\mathcal{O}_n^\infty \ne \mathcal{O}_n$.
\end{lemma}

\begin{proof}
Put $a = \sum_{d \ge 1} d^{-2}\, s_1^{\,d}$. Each $s_1^{\,d}$ is an isometry, so $\sum_d d^{-2}\|s_1^{\,d}\| < \infty$ and the series converges in norm; thus $a \in \mathcal{O}_n$. By \eqref{eq:gauge}, $\Phi_a(e^{i\theta}) = \sum_{d\ge1}d^{-2}e^{i\theta d}s_1^{\,d}$. If $a$ were smooth, the third derivative of $\Phi_a$ at $\theta = 0$ would exist in norm and equal $-i\sum_{d\ge1}d\,s_1^{\,d}$; but the terms $d\,s_1^{\,d}$ have norm $d \to \infty$, so the partial sums are not Cauchy. Hence $a \notin \mathcal{O}_n^\infty$.
\end{proof}

\subsection{Haar measure and vector-valued integration}\label{ss:haar}

Every compact Hausdorff group $G$ carries a unique regular Borel probability measure $dg$ invariant under translation, $\int_G f(hg)\,dg = \int_G f(g)\,dg$: the normalized \emph{Haar measure} \cite[Ch.~2]{Folland}. Finiteness follows from compactness, and it is what permits a bump function of integral $1$ supported in an arbitrarily small neighbourhood of the identity. For $G = \mathbb{T}$ one may take $\int_\mathbb{T} f(z)\,dz = \frac{1}{2\pi}\int_0^{2\pi}f(e^{i\theta})\,d\theta$.

The measure is scalar-valued; what is $A$-valued is the integrand. Integrals of $A$-valued functions are Bochner integrals: for a simple function $\sum_i \chi_{E_i}x_i$ with $x_i \in A$ one sets $\int_G \sum_i\chi_{E_i}x_i\,dg := \sum_i \mu(E_i)x_i \in A$, and for general $f$ one takes limits of simple approximants \cite[Ch.~II]{DiestelUhl}. No $A$-valued measure is involved; averaging of this kind over a group acting on a $C^*$-algebra is treated systematically in \cite[App.~B]{Williams}.

\begin{lemma}[properties of the integral]\label{lem:bochner}
Let $f : G \to A$ be continuous. Then $f$ is Bochner integrable and
\begin{enumerate}
\item $\int_G f(g)\,dg \in A$;
\item $\big\|\int_G f(g)\,dg\big\| \le \int_G \|f(g)\|\,dg$;
\item $T\big(\int_G f\,dg\big) = \int_G T(f(g))\,dg$ for every bounded linear $T : A \to A$;
\item $\int_G \varphi(g)\,x\,dg = \big(\int_G \varphi(g)\,dg\big)x$ for fixed $x \in A$ and scalar $\varphi \in C(G)$;
\item $\int_G f(hg)\,dg = \int_G f(g)\,dg$ for each fixed $h \in G$.
\end{enumerate}
\end{lemma}

\begin{proof}
These are standard properties of the Bochner integral; see \cite[Ch.~1]{HNVW}. Integrability follows from Bochner's criterion: $f$ is continuous on a compact metric space, hence strongly measurable with separable range, and $\int_G\|f\|\,dg \le \sup_G\|f\| < \infty$ since $\mu(G)=1$. Part~(1) then holds as $A$ is complete, and part~(2) is the norm inequality for the Bochner integral \cite[Ch.~1]{HNVW}. Part~(3), that bounded operators pass through the integral, is also standard \cite[Ch.~1]{HNVW}. Part~(4) is the special case of (3) applied to the bounded map $\CC \to A$, $\lambda \mapsto \lambda x$; and part~(5) is invariance of the (Haar) measure $\mu$: for a simple $f = \sum_i\chi_{E_i}x_i$ one has $\int_G f(hg)\,dg = \sum_i\mu(h^{-1}E_i)x_i = \sum_i\mu(E_i)x_i$, and the general case follows by approximation.
\end{proof}

\subsection{Density}

For $\varphi \in C^\infty(G)$ and $a \in A$ define the \emph{smoothed element}
\[
a_\varphi := \int_G \varphi(g)\,\alpha_g(a)\,dg \in A,
\]
well defined by Lemma~\ref{lem:bochner}, the integrand being continuous by strong continuity.

\begin{lemma}[the group variable transfers to $\varphi$]\label{lem:grouptransfer}
$\alpha_h(a_\varphi) = \int_G \varphi(h^{-1}u)\,\alpha_u(a)\,du$ for all $h \in G$.
\end{lemma}

\begin{proof}
By Lemma~\ref{lem:bochner}(3) applied to $\alpha_h$, and $\alpha_h\alpha_g = \alpha_{hg}$, we get $\alpha_h(a_\varphi) = \int_G\varphi(g)\alpha_{hg}(a)\,dg$; substituting $u = hg$ and using Lemma~\ref{lem:bochner}(5) gives the claim.
\end{proof}

The point is that on the right-hand side $h$ occurs only in the scalar factor $\varphi(h^{-1}u)$, which is smooth by choice, whereas $\alpha_u(a)$ does not involve $h$ at all.

\begin{lemma}[smoothing]\label{lem:smoothing}
$a_\varphi \in A^\infty$ for every $a \in A$ and $\varphi \in C^\infty(G)$.
\end{lemma}

\begin{proof}
It suffices to differentiate once along a one-parameter subgroup and observe that the result has the same form. Put $\psi(\theta,u) := \varphi(\exp(-\theta X)u)$ and $x_u := \alpha_u(a)$, so that by Lemma~\ref{lem:grouptransfer}
\[
\Phi(\theta) := \alpha_{\exp(\theta X)}(a_\varphi) = \int_G \psi(\theta,u)\,x_u\,du ,
\]
where $x_u$ does not depend on $\theta$ and $\|x_u\| = \|a\|$ by \eqref{eq:isom}. As $\varphi \in C^\infty(G)$ and $(\theta,u)\mapsto \exp(-\theta X)u$ is smooth, $\psi$ is smooth in $\theta$ and
$M := \sup\{|\partial_\theta^2\psi(\sigma,u)| : |\sigma|\le1,\ u \in G\} < \infty$,
being the supremum of a continuous function on a compact set. For $0 < |t| \le 1$, Taylor's theorem with Lagrange remainder applied to $\sigma \mapsto \psi(\sigma,u)$ gives
\[
R_t(u) := \frac{\psi(t,u)-\psi(0,u)}{t} - \partial_\theta\psi(0,u) = \frac{t}{2}\,\partial^2_\theta\psi(\sigma,u), \qquad |R_t(u)| \le \frac{M|t|}{2},
\]
for some $\sigma = \sigma(t,u)$ with $|\sigma| \le |t|$. Hence by Lemma~\ref{lem:bochner}(2) and $\mu(G)=1$,
\[
\Big\|\frac{\Phi(t)-\Phi(0)}{t} - \int_G \partial_\theta\psi(0,u)\,x_u\,du\Big\| = \Big\|\int_G R_t(u)\,x_u\,du\Big\| \le \frac{M\|a\|}{2}|t| \xrightarrow[t\to0]{} 0 .
\]
So $\Phi$ is differentiable, and the derivative is an integral of the same shape with $\varphi$ replaced by a derivative of $\varphi$. Induction gives derivatives of all orders, so $a_\varphi \in A^\infty$.
\end{proof}

\begin{lemma}[bump functions]\label{lem:bump}
For each $\eta > 0$ there is $\varphi \in C^\infty(G)$ with $\varphi \ge 0$, $\int_G\varphi\,dg = 1$ and $\operatorname{supp}\varphi \subseteq B_\eta(e)$.
\end{lemma}

\begin{proof}
$G$ is a smooth manifold, so there is a nonzero $\psi \in C^\infty(G)$, $\psi \ge 0$, supported in $B_\eta(e)$ (a bump function in a chart at $e$, extended by zero). Then $c := \int_G\psi\,dg > 0$, since $\psi$ is continuous, nonnegative and nonzero and Haar measure is positive on nonempty open sets; take $\varphi := c^{-1}\psi$.
\end{proof}

\begin{theorem}\label{thm:dense}
$A^\infty$ is dense in $A$.
\end{theorem}

\begin{proof}
Let $a \in A$ and $\varepsilon > 0$. Since $\Phi_a$ is continuous at $e$ with $\Phi_a(e) = a$, there is $\eta > 0$ with $\|\alpha_g(a)-a\| < \varepsilon$ whenever $d(g,e) < \eta$. Choose $\varphi$ as in Lemma~\ref{lem:bump} for this $\eta$ and put $b := a_\varphi \in A^\infty$ (Lemma~\ref{lem:smoothing}). By Lemma~\ref{lem:bochner}(4) and $\int_G\varphi = 1$ we may write $a = \int_G\varphi(g)\,a\,dg$, so by Lemma~\ref{lem:bochner}(2)
\[
\|b-a\| = \Big\|\int_G \varphi(g)\big[\alpha_g(a)-a\big]dg\Big\| \le \int_G \varphi(g)\,\|\alpha_g(a)-a\|\,dg \le \varepsilon\int_G\varphi(g)\,dg = \varepsilon ,
\]
the last inequality because the integrand vanishes off $\operatorname{supp}\varphi \subseteq B_\eta(e)$ and $\varphi \ge 0$.
\end{proof}

\subsection{The Fr\'echet structure and inverse-closedness}

\begin{theorem}\label{thm:frechetstructure}
$A^\infty$ is a unital Fr\'echet subalgebra of $A$ in the sense of Definition~\ref{def:frechet}, for the seminorms
\[
\|a\|_w := \|\delta_{j_1}\cdots\delta_{j_k}(a)\|, \qquad w = (j_1,\dots,j_k) \text{ a word in } \{1,\dots,m\},
\]
the empty word giving $\|a\|_\emptyset = \|a\|$.
\end{theorem}

\begin{proof}
\emph{Unital subalgebra.} $\alpha_g(1) = 1$, so $\Phi_1$ is constant and $1 \in A^\infty$. If $a,b \in A^\infty$ then $\Phi_{ab}(g) = \alpha_g(a)\alpha_g(b)$ is the composite of the $C^\infty$ map $g \mapsto (\Phi_a(g),\Phi_b(g))$ with the bounded bilinear multiplication $A \times A \to A$, hence $C^\infty$; so $ab \in A^\infty$. Sums and scalar multiples are clear.

\emph{Countable, Hausdorff, finer.} The words in $\{1,\dots,m\}$ form a countable set. The empty word gives the $C^*$-norm, so the topology is Hausdorff and finer than the norm topology.

\emph{Completeness.} Let $(a_n)$ be Cauchy for every seminorm $\|\cdot\|_w$. Since $\|a_n - a_k\|_w = \|\delta_w(a_n) - \delta_w(a_k)\|$, each sequence $(\delta_w(a_n))_n$ is Cauchy in the norm of $A$; as $A$ is complete it has a limit, say $a^{(w)}$, and we set $a := a^{(\emptyset)} = \lim_n a_n$. These limits are a priori unrelated elements of $A$, and completeness comes down to showing that $a^{(w)}$ is in fact the derivative $\delta_w(a)$ --- so that $a$ is smooth and $a_n \to a$ in the topology. This is exactly what closedness of the derivations provides. We prove $a \in \operatorname{dom}(\delta_w)$ with $\delta_w(a) = a^{(w)}$ by induction on the word length $|w|$; the case $w = \emptyset$ is the definition of $a$. Assume the claim for $w$ and let $w' = (j,w)$. Then $\delta_w(a_n) \to \delta_w(a)$ (inductive hypothesis) and $\delta_j(\delta_w(a_n)) = \delta_{w'}(a_n) \to a^{(w')}$; applying closedness of $\delta_j$ (Lemma~\ref{lem:der}) to the sequence $\delta_w(a_n)$ gives $\delta_w(a) \in \operatorname{dom}(\delta_j)$ and $\delta_{w'}(a) = a^{(w')}$, which is the claim for $w'$. Each closedness step thus feeds the next, so the induction runs over all words. By \eqref{eq:domdesc}, $a$ lies in every iterated domain, i.e.\ $a \in A^\infty$; and $\|a_n - a\|_w = \|\delta_w(a_n) - a^{(w)}\| \to 0$ for every $w$, so $a_n \to a$ in the topology.

\emph{Joint continuity.} Iterating the Leibniz rule of Lemma~\ref{lem:der} one letter at a time --- each $\delta_j$ landing on either factor --- gives $\delta_w(ab) = \sum_{(u,v)}\delta_u(a)\delta_v(b)$, the sum over the $2^{|w|}$ ways of distributing the letters of $w$ in order between $u$ and $v$; hence
\begin{equation}\label{eq:leibnizbound}
\|ab\|_w \le \sum_{(u,v)}\|a\|_u\,\|b\|_v,
\end{equation}
a finite sum of products of seminorms of $a$ and of $b$ (for $m=1$ this is $\|ab\|_k \le \sum_j\binom kj\|a\|_j\|b\|_{k-j}$). This inequality is precisely joint continuity of multiplication: fixing $(a_0,b_0)$ and writing $ab - a_0b_0 = (a-a_0)b_0 + a_0(b-b_0) + (a-a_0)(b-b_0)$, the triangle inequality and \eqref{eq:leibnizbound} bound $\|ab - a_0 b_0\|_w$ by a finite sum in which every term carries a factor $\|a - a_0\|_u$ or $\|b - b_0\|_v$; so if $a \to a_0$ and $b \to b_0$ in the topology, then $ab \to a_0 b_0$. In particular the seminorms $\|\cdot\|_w$ satisfy the bound required in Definition~\ref{def:frechet}.
\end{proof}

\begin{remark}[the index set]\label{rem:words}
Each iterated derivative is specified by choosing, at each stage, one of $\delta_1,\dots,\delta_m$ --- that is, by a finite word in $\{1,\dots,m\}$, where $m = \dim\g$. A \emph{basis} of $\g$ suffices, since $X \mapsto \delta_X$ is linear; indexing by all of $\g$ would give an uncountable, redundant family. \emph{Words} rather than multi-indices, because order matters: $X \mapsto \delta_X$ is a Lie algebra homomorphism, so $[\delta_i,\delta_j] = \delta_{[X_i,X_j]} \ne 0$ for non-abelian $\g$. \emph{Finite} words, since each seminorm measures a derivative of finite order and smoothness asks that all finite orders exist. As there are $m^k$ words of length $k$, the index set is countable, as Definition~\ref{def:frechet} requires.
\end{remark}

\begin{theorem}\label{thm:invclosed}
If $a \in A^\infty$ is invertible in $A$, then $a^{-1} \in A^\infty$.
\end{theorem}

The natural argument --- differentiate $a^{-1}a = 1$ to get $\delta(a^{-1}) = -a^{-1}\delta(a)a^{-1}$ --- is circular as it stands, since applying $\delta$ to $a^{-1}$ presupposes $a^{-1} \in \operatorname{dom}(\delta)$, which is what is to be proved. We argue with orbit maps instead.

\begin{proof}
Each $\alpha_g$ is an algebra automorphism, so $\alpha_g(a^{-1}) = \alpha_g(a)^{-1}$ and hence
\[
\Phi_{a^{-1}} = \iota \circ \Phi_a, \qquad \iota : \Inv(A) \to \Inv(A),\quad \iota(x) = x^{-1}.
\]
Now $\Phi_a$ is $C^\infty$ because $a \in A^\infty$, and takes values in the open set $\Inv(A)$. The map $\iota$ is $C^\infty$ there: for $\|h\| < \|x_0^{-1}\|^{-1}$ the Neumann series gives $(x_0+h)^{-1} = \sum_{k\ge0}(-1)^k(x_0^{-1}h)^kx_0^{-1}$, a norm-convergent power series in $h$, so $\iota$ is analytic with derivative $h \mapsto -x_0^{-1}hx_0^{-1}$. A composite of $C^\infty$ maps is $C^\infty$, so $\Phi_{a^{-1}}$ is $C^\infty$.
\end{proof}

\begin{remark}
Differentiating $\alpha_g(a^{-1}) = \alpha_g(a)^{-1}$ at $g = e$ now recovers $\delta_X(a^{-1}) = -a^{-1}\delta_X(a)a^{-1}$ legitimately, both sides being known to exist.
\end{remark}

\subsection{The hypotheses are satisfied}

\begin{corollary}[the hypothesis class is populated]\label{cor:examples}
Let $A$ be a unital purely infinite simple $C^*$-algebra carrying a strongly continuous action of a compact Lie group. Then $A^\infty$ is a dense, hfc-closed, unital subring of $A$ which is purely infinite simple as a ring.
\end{corollary}

\begin{proof}
$A^\infty$ is dense by Theorem~\ref{thm:dense}, a Fr\'echet subalgebra by Theorem~\ref{thm:frechetstructure}, and inverse-closed by Theorem~\ref{thm:invclosed}; hence hfc-closed by Theorem~\ref{thm:schweitzer}. Being a dense inverse-closed unital subring of a purely infinite simple $C^*$-algebra, it is purely infinite simple as a ring by Theorem~\ref{thm:converse}.
\end{proof}

\begin{example}[the Cuntz algebra]\label{ex:cuntz}
Let $\mathcal{O}_n$ ($2 \le n < \infty$) be the Cuntz algebra, generated by isometries $s_1,\dots,s_n$ with $s_i^*s_j = \delta_{ij}1$ and $\sum_i s_is_i^* = 1$, and let $\mathbb{T}$ act by the gauge action $\alpha_z(s_i) = zs_i$. Corollary~\ref{cor:examples} applies: $\mathcal{O}_n$ is unital and purely infinite simple, and the action is strongly continuous, so
\[
\mathcal{O}_n^\infty = \{a \in \mathcal{O}_n : z \mapsto \alpha_z(a) \text{ is } C^\infty\}
\]
is a dense hfc-closed unital subalgebra of $\mathcal{O}_n$ which is purely infinite simple as a ring. Here $m = \dim\mathbb{T} = 1$, so there is a single derivation $\delta$, acting on monomials by $\delta(s_\mu s_\nu^*) = i(|\mu|-|\nu|)s_\mu s_\nu^*$, and the seminorms are $\|a\|_k = \|\delta^k(a)\|$.

The inclusion $\mathcal{O}_n^\infty \subseteq \mathcal{O}_n$ is proper (Lemma~\ref{lem:proper}): the element $a = \sum_{d \ge 1} d^{-2}s_1^{\,d}$ converges in norm, but its orbit map is not three times differentiable, since the formal third derivative $-i\sum_{d\ge1} d\,s_1^{\,d}$ has terms of norm $d \to \infty$. Note also that the Leavitt algebra $L_n$, the dense $*$-subalgebra spanned by the monomials, satisfies $L_n \subseteq \mathcal{O}_n^\infty$ --- each monomial has orbit map $\theta \mapsto e^{i\theta d}s_\mu s_\nu^*$, manifestly smooth --- so for this action density of $\mathcal{O}_n^\infty$ also follows directly from density of $L_n$, without recourse to the G\aa{}rding argument in the proof of Corollary~\ref{cor:examples}. The general argument is needed when no such distinguished dense subalgebra of smooth elements is available.
\end{example}

\begin{remark}\label{rem:three}
Example~\ref{ex:cuntz} exhibits the chain
\[
L_n \ \subsetneq\ \mathcal{O}_n^\infty \ \subsetneq\ \mathcal{O}_n ,
\]
the first inclusion being proper because $\sum_{d\ge1}2^{-d}s_1^{\,d}$ is smooth but not a finite sum of monomials. Informally $L_n$ is the ``polynomial'' layer, $\mathcal{O}_n^\infty$ the ``smooth'' layer and $\mathcal{O}_n$ the ``continuous'' layer, in exact analogy with trigonometric polynomials inside $C^\infty(\mathbb{T})$ inside $C(\mathbb{T})$.
\end{remark}

\begin{remark}\label{rem:provenance}
Smooth vectors for a compact group action, and the Fr\'echet structure they carry, go back to G\aa{}rding; they are developed for $C^*$-algebras in \cite{BC}. The circle action on $\mathcal{O}_n$ and its fixed-point core are due to Cuntz \cite{Cuntz77}, and the term ``gauge action'' comes from the Cuntz--Krieger and graph-algebra literature \cite{Raeburn}.
\end{remark}

\begin{remark}\label{rem:regimes}
A dense subring can have two desirable properties, and for the transfer problem they pull in opposite directions. One is being \emph{hfc-closed}, which Theorem~\ref{thm:transfer} requires. The other is having its pure infiniteness verifiable \emph{on its own}, without reference to the completion.

The known examples have exactly one of the two. Combinatorial subalgebras --- the Leavitt algebra $L_n$, and Steinberg algebras generally --- have exact generators, so their pure infiniteness as a ring is easy to check directly; but they are not hfc-closed (Example~\ref{ex:laurent}), and Theorem~\ref{thm:transfer} does not apply to them. These are the algebras handled, by different methods, in the groupoid setting of Brown--Clark--an~Huef \cite{BCH}. Smooth subalgebras such as $\mathcal{O}_n^\infty$ are the reverse: they are hfc-closed, so the theorem applies, but their pure infiniteness is known only because the completion's is (Theorem~\ref{thm:converse}).

The question left open is whether a single subring can have both properties at once: a dense hfc-closed subring, purely infinite as a ring for a reason internal to it rather than inherited from the completion. Theorem~\ref{thm:transfer} would then yield the pure infiniteness of the completion from purely ring-theoretic data.
\end{remark}

\section*{Acknowledgements}
The author gratefully acknowledges the financial support and hospitality provided by Dr. S Barik during the visit at ISI Bangalore, where most part of this article was written, which is funded through his INSPIRE Faculty Fellowship Research Grant (DST/INSPIRE/Faculty/\\2023/IFA-23-MA-201) from the Department of Science and Technology (DST), Government of India. The author is grateful to Prof. Ping W Ng for his constant encouragement and support during the course of this work.

\end{document}